\documentclass[a4paper,12pt]{amsart}

\usepackage{amssymb}
\usepackage{enumerate}
\usepackage[OT4]{fontenc}
\usepackage{hyperref}
\hypersetup{
    colorlinks,
    linkcolor=blue,
    citecolor=blue,
    urlcolor=blue,
    breaklinks=true
}
\usepackage[margin=1in]{geometry}

\usepackage{comment}
\usepackage{pgf,tikz}
\usepackage{mathrsfs}
\usetikzlibrary{arrows}

\theoremstyle{definition}
\newtheorem{definition}{Definition}
\newtheorem{example}[definition]{Example}
\newtheorem{remark}[definition]{Remark}

\theoremstyle{plain}
\newtheorem{theorem}[definition]{Theorem}
\newtheorem*{theorem*}{Main Theorem}
\newtheorem{proposition}[definition]{Proposition}
\newtheorem{lemma}[definition]{Lemma}

\newtheorem{corollary}[definition]{Corollary}

\DeclareMathOperator{\rank}{rank}

\def\PP{\mathbb{P}}

\def\CC{\mathbb{C}}

\def\a0n{a_0,\ldots,a_n}
\def\x0n{x_0,\ldots,x_n}
\def\b0n{b_0,\ldots,b_n}
\def\y0n{y_0,\ldots,y_n}
\def\p0n{p_0,\ldots,p_n}

\let\hat\widehat

\begin{document}

\title{A matrixwise approach to unexpected hypersurfaces}
\author{Marcin~Dumnicki, {\L}ucja~Farnik, Brian~Harbourne, \\
	Grzegorz~Malara, Justyna~Szpond, Halszka~Tutaj-Gasi\'nska}
\renewcommand{\shortauthors}{Dumnicki et al.}
\date{}
\begin{abstract}
	The aim of this note is to give a generalization of some results concerning unexpected hypersurfaces. 
	Unexpected hypersurfaces occur when the actual dimension of the space of forms satisfying certain vanishing 
	data is positive and the imposed vanishing conditions are not independent.
	The first instance studied were unexpected curves
	in the paper by Cook II, Harbourne,  Migliore,  Nagel. Unexpected hypersurfaces were
	then investigated by Bauer, Malara, Szpond and Szemberg, followed by Harbourne, Migliore, Nagel and Teitler
	who introduced the notion of BMSS duality and showed it holds in some cases
	(such as certain plane curves and, in higher dimensions, for certain cones). 
	They ask to what extent such a duality holds in general.
	In this paper, working over a field of characteristic zero, we study hypersurfaces
	in $\PP^n\times\PP^n$ defined by determinants. We apply our results to unexpected hypersurfaces
	in the case that the actual dimension is 1 (i.e., there is a unique unexpected hypersurface).
	In this case, we show that a version of BMSS duality always holds, as a consequence 
	of fundamental properties of determinants.

\end{abstract}


\maketitle

\parskip=.2cm

\section{Introduction}

Motivated by the results of \cite{DIV} concerning the failure of Lefschetz properties, the authors  of \cite{C-H-M-N} introduced the notion of an {\it unexpected} curve, i.e., a plane curve of degree $d$ passing through a given set of points $Z$ and having a general point $B$ of multiplicity $m=d-1$, such that the conditions imposed by $mB$ on the forms of degree $d$ vanishing on $Z$ are not independent. Since then the notion has been generalized to unexpected hypersurfaces and
quite a few papers on the subject have appeared; see e.g. \cite{DiMMO}, \cite{DHRSTG}, \cite{FGST}, \cite{H-M-TG}, \cite{Szpond2018Togliatti}, \cite{Szpond2018}.
Of particular relevance to us are the papers \cite{BMSS} and \cite{H-M-N-T}, from which comes the notion of BMSS duality.

What we will refer to as BMSS duality is as follows. Given integers $d\geq m>0$
and points $P_1,\ldots,P_r$ in $\PP^n$
and a general point $B=(b_0,\dots,b_n)\in\PP^n$, assume that there is a unique hypersurface $H_B$ of degree $d$
containing each point $P_i$ and having multiplicity $m$ at $B$.
Then there is a bihomogeneous form $F(a_0,\dots,a_n,x_0,\dots,x_n)\in \CC[\a0n][\x0n]$ 
such that $F(B,\x0n)=0$ defines the hypersurface $H_B$.
One version of BMSS duality is the fact that $F(B,\x0n)$ and $F(\x0n,B)$ have the same tangent cone
at $B$. This is proved in special cases in \cite{H-M-N-T}; 
proving this in general is one of our main results, see Corollary \ref{BMSSCor} and
Remark \ref{remTC}. In the special cases studied
in \cite{H-M-N-T} even more can be said, about the bidegree of $F(a_0,\dots,a_n,x_0,\dots,x_n)$, as we discuss in more detail below.
We partially address this too, see Theorem \ref{k-mult} and Example \ref{exb3r}.

In more detail, we now  discuss one of these special cases, indeed, the example 
for which BMSS duality was first recognized.
It comes from the arrangement of lines classically denoted by $B_3$ (or $A(9,1)$ in Gr\"unbaum's notation, \cite{Grunbaum}),
given by the linear factors of the polynomial
$$
f=xyz(x+y)(x-y)(x+z)(x-z)(y+z)(y-z).
$$
These factors correspond dually to the points of a subscheme $Z\subset \PP^2$ (see Figure \ref{Fig1}):
$$\begin{array}{lll}
P_1=(1:0:0),& P_2=(0:1:0),& P_3=(0:0:1),\\
P_4=(1:1:0),& P_5=(1:-1:0),& P_6=(1:0:1),\\
P_7=(1:0:-1),& P_8=(0:1:1),& P_9=(0:1:-1).
\end{array}$$
It is not difficult to observe  that the points in $Z$
impose independent conditions on quartics in $\PP^2$, see \cite{C-H-M-N}.

Take  the linear space
of  quartic curves in $\PP^2$ vanishing at $P_1,\dots,P_9$.
As the points impose  independent conditions on the space of quartics,
the dimension of this subspace is $6$, 
but a triple point typically imposes 6 conditions
so we expect that a curve in the system cannot have a (general) triple point.
Nonetheless, there exists
for any choice of an additional point $B=(a:b:c)$
an unexpected quartic in this subspace with a triple point at $B$.
As observed in \cite{BMSS}, 
this quartic can be defined by a bihomogeneous form $G\in \mathbb{C}[a,b,c,x,y,z]$
which can be viewed in two ways, first as a quartic in the variables $x,y,z$ with coefficients 
parameterized by the triple point $B=(a:b:c)$, so as
\begin{align*}
Q_B(x,y,z) =  3a(b^2-c^2) x^2yz +3b(c^2-a^2) xy^2z 
+3c(a^2-b^2) xyz^2 \\
+a^3 y^3z -a^3 yz^3 +b^3 xz^3
-b^3 x^3z + 
c^3 x^3y -c^3 xy^3
\in \mathbb{C}[a,b,c][x,y,z],
\end{align*}
and second, as a cubic in the variables $a,b,c$ with coefficients parameterized by $S=(x,y,z)$, namely as
\begin{align*}
Q_S(a,b,c)=
yz(y^2-z^2) a^3 +xz(z^2-x^2) b^3 +xy(x^2-y^2) c^3+ 
3x^2yzab^2\\  
-3xy^2z a^2b + 3xyz^2 a^2c -3x^2yz ac^2 
+3xy^2z bc^2- 3xyz^2 b^2c \in \mathbb{C}[x,y,z][a,b,c],
\end{align*}
and these cubic curves have a triple point at $S$ (consequently it is composed of three lines).

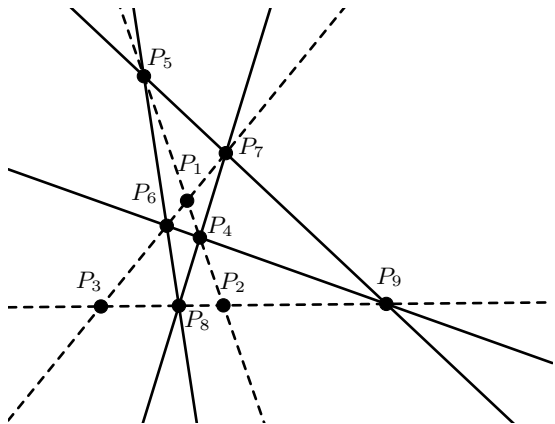
\begin{figure}
\begin{tikzpicture}[line cap=round,line join=round,>=triangle 45,x=0.75cm,y=0.75cm]
\clip(-2.570516224609383,-0.95182393413239) rectangle (6.973081328879615,6.353964537848843);
\draw [line width=1.pt,domain=-2.570516224609383:6.973081328879615] plot(\x,{(-0.8366430877314829--1.4150851544828515*\x)/-0.21390822102647752});
\draw [line width=1.pt,domain=-2.570516224609383:6.973081328879615] plot(\x,{(-0.23561695137758343--2.698534480641717*\x)/0.8227239270249136});
\draw [line width=1.pt,domain=-2.570516224609383:6.973081328879615] plot(\x,{(-9.993809155662227--1.382176197401855*\x)/-3.850347978476596});
\draw [line width=1.pt,domain=-2.570516224609383:6.973081328879615] plot(\x,{(-14.025854692755985--2.6656255235607205*\x)/-2.8137158304252043});
\draw [dashed,line width=1.pt,domain=-2.570516224609383:6.973081328879615] plot(\x,{(--4.527702289724097-2.8529944648199077*\x)/0.9813788094429696});
\draw [dashed,line width=1.pt,domain=-2.570516224609383:6.973081328879615] plot(\x,{(--2.342489261857367--1.2834493261588653*\x)/1.0366321480513911});
\draw [dashed,line width=1.pt,domain=-2.570516224609383:6.973081328879615] plot(\x,{(--4.004254655526826--0.032908957080996615*\x)/3.6364397574501184});
\begin{scriptsize}
\draw [fill=black] (0.42419886976130233,1.104985883429894) circle (2.5pt);
\draw[color=black] (0.75,0.85) node {$P_8$};
\draw [fill=black] (0.2102906487348248,2.5200710379127456) circle (2.5pt);
\draw[color=black] (-0.17638959696688458,3.095977786830185) node {$P_6$};
\draw [fill=black] (1.246922796786216,3.803520364071611) circle (2.5pt);
\draw[color=black] (1.7,3.9022472353146003) node {$P_7$};
\draw [fill=black] (4.0606386272114205,1.1378948405108906) circle (2.5pt);
\draw[color=black] (4.118229302103164,1.6150747181853404) node {$P_9$};
\draw [fill=black] (-0.18941214141179333,5.164258726574989) circle (2.5pt);
\draw[color=black] (.1,5.481877175202435) node {$P_5$};
\draw [fill=black] (0.7919666680311763,2.311264261755081) circle (2.5pt);
\draw[color=black] (1.15,2.5036165593722473) node {$P_4$};
\draw [fill=black] (0.5678588494328471,2.9627745244912496) circle (2.5pt);
\draw[color=black] (0.6298798515175307,3.6060666215856316) node {$P_1$};
\draw [fill=black] (-0.9426536617907761,1.092616177262002) circle (2.5pt);
\draw[color=black] (-1.1472038308562826,1.5328023254828491) node {$P_3$};
\draw [fill=black] (1.204475775283025,1.1120472128916288) circle (2.5pt);
\draw[color=black] (1.370331385839953,1.5163478469423508) node {$P_2$};
\end{scriptsize}
\end{tikzpicture}
\caption{Collinearities of the points $P_i$ of $Z\subset\PP^2$, with coordinates axes (dashed).}
\label{Fig1}
\end{figure}
The form $G$ has bidegree $(3,4)$ and defines a hypersurface in $\PP^2\times \PP^2$.
The fibers with respect to the first projection are quartics $G(B,\cdot)=Q_B(\cdot)=0$ with a triple point at $B$,
while the fibers with respect to the second projection are cubics $G(\cdot,S)=Q_S(\cdot)=0$ with a triple point at $S$.
The connection between $G(\cdot,S)$ and $G(B,\cdot)$ is that they have the same tangent cone at $B=S$. 

More generally, consider
sets of points $Z$ of $\PP^2$ such that the set of lines dual to $Z$ is what is known as a free arrangement,
with the assumption that $Z$ admits a unique, irreducible unexpected curve of degree $d$ with a general point $P=(a,b,c)$
of multiplicity $m=d-1$. In this case, one of the main results of \cite{H-M-N-T} 
is a proof of what they call BMSS duality, namely
that there is a bihomogeneous form $G(a,b,c,x,y,z)$
of bidegree $(m,d)$ such that $G(B,\cdot)$ defines the degree $d$ unexpected curves with multiplicity $m$ at $B$
and $G(\cdot,S)$ defines curves of degree $m$ of multiplicity $m$ at $S$, and that
the tangent cones of $G(\cdot,S)$ and $G(B,\cdot)$ at $B=S$ are the same.
The example of the unexpected plane quartics from \cite{BMSS} is a special case of this.

Here, at the cost of having a less explicit result for the bidegree of $G$, 
we show that a version of this result (see Corollary \ref{BMSSCor} and Remark \ref{remTC}) holds in all dimensions
and without the assumption that $Z$ comes from a free hyperplane arrangement.
Our idea is to work with the largest subset of $Z$ such that what we get is expected.
For example, consider the situation above of the unexpected plane quartics 
through a set $Z$ of nine points, but replace $Z$ with any eight out of the nine points. 
Then we still get quartic curves but they are not unexpected, since the expected dimension of the system 
of quartics passing through these eight points and through one general point $B$ 
with multiplicity three is one, which is exactly what we get.
(What is unexpected is that there is a ninth point, this being the point which we excluded from the original nine points of $Z$,
which lies on every such quartic. For a general choice of eight points, there is still a unique quartic
through the eight points with a general triple point, but there is no ninth point lying on every such quartic.)

This idea is the starting point for our paper. More generally, 
consider a set $Z$ of points in $\PP^n$ (where $\PP^n$ has coordinates $(x_0,\ldots,x_n)$) such that $Z$ and a general point 
$B=(a_0,\ldots,a_n)$ with multiplicity $m$ impose independent conditions on forms of degree $d$,
with the affine dimension of the system 
of forms of degree $d$ vanishing on $Z$ and on
$mB$ being 1, hence
$\binom{d+n}{n}=\binom{m+n-1}{n}+r+1$.
Then the equation of the hypersurface is a determinant $F$ of a suitable interpolation matrix $M$ 
involving two groups of coordinates, $(a_0,\ldots,a_n)$, $(x_0,\ldots,x_n)$. As we will see, when $F$ 
is not identically zero it is a bihomogeneous polynomial (of bidegree $(\binom{m+n-1}{n} (d-m+1),d)$)
and so defines a variety in $\PP^n\times\PP^n$ with two 
natural projections. We investigate the properties of this variety and we show in particular that 
the fiber over a fixed 
point $S$, defined by the equation $F(\a0n,S)=0$, is a hypersurface in $\PP^n$ of degree $\binom{m+n-1}{n} (d-m+1)$
which vanishes at both $Z$ and $S$ with multiplicity at least $m$ (Theorem \ref{m-at-Z-and_Pv1}).  We also
prove that the fiber over a given point $B\in \PP^n$, defined by $F(B,\cdot)=0$, has the same tangent cone 
at $S=B$ as does $F(\a0n,S)=0$ 
since $F(\cdot,S)$ and $F(B,\cdot)$ have the same partial derivatives of order $m$ at $S=B$ (Theorem \ref{tangent}).

To relate this to our discussion above of BMSS duality for the unexpected plane quartics,
note that $F$ factors as $F=HG$, where $H\in \mathbb{C}[a_0,\dots,a_n]$
is a form (possibly a constant) describing the codimension $1$ part of the locus of points $B$ such that $F(B,\cdot)\equiv0$.
In the case of the unexpected plane quartics discussed above,
we have $n=2$, $d=4$ and $m=3$ so $F$ has bidegree $(\binom{m+n-1}{n} (d-m+1),d)=(12,4)$,
but $H=ab^3c^3(a+b-c)(a-b+c)$ (see Example \ref{exb3r}),
so in this case $H$ has bidegree $(9,0)$ and $G$ has bidegree $(3,4)$ 
which gives the bidegree of $F$ to be $(12,4)$, as asserted.
While we do not have an explicit expression for the bidegree of $H$ (or for $G$),
our results show that BMSS duality holds for $F$ (and hence for $G$, whatever its bidegree is).
 
\section{Definitions and first properties}
Assume that we have a set $Z=\{P_1, \ldots ,P_r\}$ of $r$  distinct points in projective space $\PP^n$ over any field $\mathbb{F}$ of characteristic $0$. Let $d$ and $m$ be positive integers such that the following equality is satisfied:
\begin{equation}\label{dim}
 \binom{d+n}{n}=\binom{m+n-1}{n}+r+1.
 \end{equation}

 In $\PP^n\times \PP^n$ define a variety 
$$ V:=\{(B,P)\in \PP^n\times\PP^n : 
   [I(mB)\cap I(P)\cap I(Z)]_d\neq 0\},$$
where $I(mB)\cap I(P)\cap I(Z)$ is the ideal generated by 
homogeneous polynomials vanishing on $B$ with multiplicity $m$ and also vanishing on $Z$ and $P$.

Let us introduce the following notation. 	Fix the vector $w$ of all the monomials (in any order) in $(\a0n)$ of degree $d$, for example
$$w:=(a_{0}^d, a_0^{d-1}a_1, a_0^{d-1}a_2,\ldots,a_{n-1}a_n^{d-1},a_n^{d}).$$

 Assume that the points $P_i\in Z$ have coordinates $(p_{i0},\ldots,p_{in})$, for $i=1,\ldots,r$. 
 Define a matrix $M$, which we call an interpolation matrix, as follows. The first $\binom{m+n-1}{n} $ rows are all the partial derivatives of $w$ of order $m-1$.
The next $r$ rows are $w(p_{10},\ldots,p_{1n}),\ldots,$  $w(p_{r0},\ldots,p_{rn})$ and the last row is $w(\x0n)$, 
in the variables $x_i$. Hence the  matrix $M$ is a square matrix of the form: 
$$
M=\left[\begin{array}{c}
\frac{\partial^{m-1}}{\partial a_0^{m-1}}w \\
\frac{\partial^{m-1}}{\partial a_0^{m-2}\partial a_1}w \\
\vdots\\
\frac{\partial^{m-1}}{\partial a_n^{m-1}}w \\
w(P_1)\\
\vdots\\
w(P_r)\\
w(\x0n)
\end{array}\right]=
$$
{\tiny	$$
	\left[\begin{array}{cccccc}
	\frac{\partial^{m-1}}{\partial a_0^{m-1} }(a_0^d)&	\frac{\partial^{m-1}}{\partial a_0^{m-1} }(a_0^{d-1}a_1) &	\frac{\partial^{m-1}}{\partial a_0^{m-1} }(a_0^{d-1}a_2) &
	\dots&
	\frac{\partial^{m-1}}{\partial a_0^{m-1} }(a_{n-1}a_n^{d-1}) &
	\frac{\partial^{m-1}}{\partial a_0^{m-1} }(a_n^d)  \\
	
	\frac{\partial^{m-1}}{\partial a_0^{m-2}\partial a_1}(a_0^d)&	\frac{\partial^{m-1}}{\partial a_0^{m-2}\partial a_1 }(a_0^{d-1}a_1) &	\frac{\partial^{m-1}}{\partial a_0^{m-2}\partial a_1 }(a_0^{d-1}a_2) &
	\dots&
	\frac{\partial^{m-1}}{\partial a_0^{m-2}\partial a_1 }(a_{n-1}a_n^{d-1}) &
	\frac{\partial^{m-1}}{\partial a_0^{m-2}\partial a_1 }(a_n^d) \\
	\vdots&\vdots&\vdots& &\vdots&\vdots\\
	\frac{\partial^{m-1}}{\partial a_n^{m-1}}(a_0^d)&	\frac{\partial^{m-1}}{\partial a_n^{m-1} }(a_0^{d-1}a_1) &	\frac{\partial^{m-1}}{\partial a_n^{m-1} }(a_0^{d-1}a_2) &
	\dots&
	\frac{\partial^{m-1}}{\partial a_n^{m-1} }(a_{n-1}a_n^{d-1}) &
	\frac{\partial^{m-1}}{\partial a_n^{m-1}}(a_n^d)
	\\
	p_{10}^d& p_{10}^{d-1}p_{11} & p_{10}^{d-1}p_{12}&\ldots&p_{1n-1}p_{1n}^{d-1}&p_{1n}^{d}\\
	\vdots&\vdots&\vdots& &\vdots&\vdots\\
	p_{r0}^d& p_{r0}^{d-1}p_{r1} & p_{r0}^{d-1}p_{r2}&\ldots&p_{rn-1}p_{rn}^{d-1}&p_{rn}^{d}\\
	x_{0}^d& x_0^{d-1}x_1 & x_0^{d-1}x_2&\ldots&x_{n-1}x_n^{d-1}&x_n^{d}
	\end{array}\right]
	$$}

Now we  characterize the set $V$.

\begin{proposition}\label{pierwsze} Let $d,m, r$, $Z$ and $M$ be as above.
Let  $F(\a0n,\x0n):=\det( M (\a0n,\x0n))$.
Then either $F$ vanishes identically, or $F$ is bihomogeneous with respect to $(\a0n)$ and $(\x0n)$ with bidegree 
	$$\left(\binom{m+n-1}{n}(d-m+1),d\right).$$
In either case we have
 $$V=V(F)=\{(\a0n,\x0n)\ :\ F(\a0n,\x0n)=0 \}.$$
\end{proposition}

\begin{proof} First, we note that an example of the above matrix is given in Example \ref*{exb3r}, and having a look at it may make following the proof easier. 
	
Let $w_j$ denote the $j$th entry of $w$. Take the equation $f=\sum_j \theta_jw_j=0$, 
where $\theta_j\in \mathbb{F}$ and $\Theta=(\theta_1,\dots,\theta_{\binom{d+n}{n}})$,  
of a hypersurface $C$ of degree $d$. The conditions for 
$C$ to vanish with multiplicity $m$ at a point $B=(\a0n)$, to pass through $Z$ and 
$P=(\x0n)$ are given by (differentiating if necessary and) evaluating $f$ at 
coordinates of these points. This leads to a system of linear equations represented by $M$.
Thus $C$  exists if and only if there exists  a non-zero solution  of the system  
$M\Theta=0$. Since $M$ is a square matrix, by (\ref{dim}), this is equivalent to $\det M=0$.  
\end{proof}

\begin{remark}\label{morethanone}
Fix a point $B=(\b0n)$. We observe that $F(B,\x0n) \equiv 0$ if and only if there exists more than one hypersurface of degree $d$
in $I(mB)\cap I(Z)$.
\end{remark}

\begin{remark}
Fix a point $B=(\b0n)$ such that $F_B = F(B,\x0n) \not\equiv 0$. Then by Remark \ref{morethanone}, $F_B=0$ defines a unique hypersurface of degree $d$ in $I(mB)\cap I(Z)$. Indeed, by definition it describes the locus of points $S$, such that there exists a form in $I(mB)\cap I(Z)\cap I(S)$ of degree $d$. So, in fact,
it describes points lying on the hypersurface defined by this form.
\end{remark}

\begin{remark}
Fix a point $S=(\y0n)$. Observe that $F(\a0n,S) \equiv 0$ if and only if for each point $B$ 
there exists a nonzero form of degree $d$ in $I(mB)\cap I(Z)\cap I(S)$. 
It is reasonable to consider two separate cases: either (a) for each $B$ there is a positive 
dimensional family of hypersurfaces defined by forms in $I(mB)\cap I(Z)$ of degree $d$, 
hence there is one passing through $S$, or (b) for a general $B$ there is a unique form in 
$I(mB)\cap I(Z)$ of degree $d$ but this unique form, which depends on $B$, nevertheless always vanishes at $S$. 
In both cases we have $F(\a0n,S) \equiv 0$ for each $S \in Z$, but in the second case, if the set
$$Z'=\{ S \in \PP^n : F(\a0n,S) \equiv 0 \}$$
is strictly larger than $Z$, then we say $Z'$ admits an ``unexpected'' hypersurface of degree $d$.
\end{remark}

\begin{remark}
Fix a point $S=(\y0n)$ such that $F_S = F(\a0n,S) \not\equiv 0$. Then $F_S=0$ 
describes the locus of such points $B$, for which a hypersurface in $I(mB)\cap I(Z)$ of degree $\binom{m+n-1}{n}(d-m+1)$
passes through a fixed $S$. This hypersurface is at the center of our research.
\end{remark}

\section{Fixing a point in the first set of variables}

Proposition \ref{m-at-B}
follows from Proposition \ref{pierwsze}; however, we present here a direct proof, with 
an argument which we will use often in the sequel, without referring to a system of linear equations.

\begin{proposition}\label{m-at-B}
Let $F$ be as in Proposition \ref{pierwsze} and
fix a point $B=(\b0n)$. Assume that  $F(\b0n,\x0n)\not\equiv 0$. Then $F(B,\x0n)=0$ gives an equation of a hypersurface $C$ (in variables $\x0n$) with $Z\subset C$ and for which $B$ is a point of multiplicity at least $m$.
\end{proposition}

\begin{proof} 	
We observe that $F(B,P_i)=0$, as it is the determinant of the matrix $M$ with $w(P_i)$ in the last row (we substitute $p_{i0},\dots,p_{in}$ for $\x0n$), so two rows repeat and the determinant is indeed zero. Thus $Z\subset C$.
	
To see that $F(B,\x0n)$ vanishes with multiplicity at least $m$ at $B$ we have to compute all partial derivatives of $F(B,\x0n)$ of order $m-1$, and check they are zero at $P$. Note that $F(B,\x0n)$ is a polynomial in $\mathbb{F}[\x0n]$, derivatives are taken with respect to $\x0n$ and evaluating at $B$ is given by substituting $b_j$ for $x_j$.
	
As determinants are multi-linear functions on rows of a matrix, the partial derivative of the determinant is the sum of the determinants of matrices with appropriate derivatives of rows taken. Thus for a first order partial differential operator $\partial$ we have
\begin{equation}\label{multidet}
\partial \det \left[ \begin{array}{c} \text{row}_1 \\ \text{row}_2 \\ \text{row}_3 \\ \vdots \end{array}\right] =
\det \left[ \begin{array}{c} \partial\text{row}_1 \\ \text{row}_2 \\ \text{row}_3 \\ \vdots  \end{array}\right] +
\det \left[ \begin{array}{c} \text{row}_1 \\ \partial\text{row}_2 \\ \text{row}_3 \\ \vdots \end{array}\right] + 
\det \left[ \begin{array}{c} \text{row}_1 \\ \text{row}_2 \\ \partial\text{row}_3 \\ \vdots \end{array}\right] + \cdots .
\end{equation}
An analogous formula may be written for higher order partial differential operators.

Note that if we take any derivative (with respect to $\x0n$) of any row of the 
matrix except the last one, we get a row of zeros (since $B$ and $P_i$ are fixed). 
Thus, the only possibility for getting a nonzero term is to take all the required 
partial derivatives from the last row. But then, substituting $(\b0n)$ for $(\x0n)$ the row we get 
is exactly one from the first $\binom{m+n-1}{n}$ rows, so the determinant is  zero.  
\end{proof}

We now want to rewrite the matrix $M$. Let $\partial_i$ denote $\partial/\partial_{a_i}$.
The top $\binom{m+n-1}{n}$ rows of $M$ are precisely the order $m-1$ partials
$\partial_{i_{m-1}}\cdots\partial_{i_1}w(a_0,\dots,a_n)$ of the vector $w(a_0,\dots,a_n)$ of monomials
of degree $d$ in the $a_i$. We may assume, for simplicity, that in every case we write $\partial_{i_{m-1}}\cdots\partial_{i_1}$ 
in such a way that all instances of the partial $\partial_0$ which occur in $\partial_{i_{m-1}}\cdots\partial_{i_1}$ come at the left
(and so are applied last). 

Suppose $\partial$ represents a given partial of order $m-2$,
so $\partial_i\partial$ has order $m-1$ for each $i$ and thus $\partial_i\partial w(a_0,\dots,a_n)$ corresponds to some row 
of $M$. Denote this row by $R_i$. Note by Euler's identity that $a_0^{-1}\sum_i a_iR_i=a_0^{-1}l\partial w(a_0,\dots,a_n)$,
where $l=d-(m-2)$ is the degree of the nonzero entries of $\partial w(a_0,\dots,a_n)$.
Thus in $M$ if we replace $R_0$ by $a_0^{-1}\sum_i a_iR_i=a_0^{-1}l\partial w(a_0,\dots,a_n)$, we do not change 
the determinant, since this merely adds multiples of other rows to $R_0$. 
Proceeding in this way, for each row of $M$ which involves the partial $\partial_0$,
we can replace one occurrence of $\partial_0$ by a factor of $a_0^{-1}$ times an appropriate nonzero scalar, reducing by 1
the number of occurrences of $\partial_0$ in each row in which $\partial_0$ occurs.
We can repeat this procedure until no row involves $\partial_0$,
thereby obtaining the matrix $\hat{M}$, where each occurrence of $\partial_0$ has been replaced 
by a factor of $a_0^{-1}$ and where $c\det(\hat{M})=\det(M)$ for a nonzero scalar $c$ coming 
from the degrees which occur in Euler's identity.
By specifying a particular ordering of the rows of $M$, we may assume that $\hat{M}$ is as given in Figure \ref{Mhat}.

\begin{figure}
$$\hat{M}=\left[\begin{array}{c}
a_0^{-(m-1)}w(a_0,\dots,a_n)\\
a_0^{-(m-2)}\frac{\partial w(a_0,\dots,a_n)}{\partial a_1 } \\
a_0^{-(m-2)}\frac{\partial w(a_0,\dots,a_n)}{\partial a_2 } \\
\vdots\\
a_0^{-(m-2)}\frac{\partial w(a_0,\dots,a_n)}{\partial a_{n} }\\
a_0^{-(m-3)}\frac{\partial^2w(a_0,\dots,a_n)}{\partial a_1^2 } \\
a_0^{-(m-3)}\frac{\partial^2w(a_0,\dots,a_n)}{\partial a_1 \partial a_2} \\
\vdots\\
a_0^{-(m-3)}\frac{\partial^2w(a_0,\dots,a_n)}{\partial a_n^2 } \\
\vdots\\
a_0^0\frac{\partial^{m-1}w(a_0,\dots,a_n)}{\partial a_{1}^{m-1} } \\
\vdots\\
a_0^0\frac{\partial^{m-1}w(a_0,\dots,a_n)}{\partial a_{n}^{m-1} } \\
w(P_1)\\
\vdots\\
w(P_r)\\
w(\x0n)
\end{array}\right].
$$
\caption{A particular choice of ordering the rows of $\hat{M}$.}
\label{Mhat}
\end{figure}

\begin{theorem}\label{m-at-Z-and_Pv1}
Let $F=\det(M)$.
Fix a point $S=(\y0n)$ in $\PP^n$. Assume that $F(\a0n,S)\not\equiv 0$. 
Then $F(\a0n,S)=0$ gives a hypersurface of degree  $\binom{m+n-1}{n}(d-m+1)$ 
which vanishes with multiplicity at least $m$ at $S$ and at each point of $Z=P_1+\dots+P_r$.
\end{theorem}

\begin{proof}
Note that $F(\a0n,S)$ is the determinant of a matrix having one row each of the form 
$w(P_1),\dots,w(P_r),w(S)$, and otherwise whose rows are independent of
$P_1,\dots,P_r$ and $S$. So the proof for each $P_i$ and for $S$ is the same.
Thus it is enough to prove it for $P_1$. After a change of coordinates if need be,
we can assume that the $a_0$ coordinate for each of the points $P_1,\dots,P_r$ and $S$
is nonzero.
 
Also note that $F$ is bihomogeneous, hence gives a form defined on $\PP^n\times \PP^n$.
It suffices to show that $F\in I(P_1\times \PP^n)^m$. This will follow if, regarding
$F$ as being in $\CC[a_0,\dots,a_n][x_0,\dots,x_n]$, we show that
the coefficients of $F$, being in $\CC[a_0,\dots,a_n]$,
are in $I(P_1)^m=I(P_1)^{(m)}$, and this can be checked by showing that
all partials of $F(a_0,\dots,a_n)$ in the variables $a_i$ of order exactly $m-1$ vanish at $P_1=(1,p_{11},\dots,p_{1n})$.
And this vanishing holds if and only if all partials of $F(1,a_1,\dots,a_n)$ of order at most $m-1$ 
in the variables $a_1,\dots,a_n$ vanish at $(p_{11},\dots,p_{1n})$. 

Since, apart from a nonzero scalar factor, $M$ and $\hat{M}$ have the same determinant,
we may replace $F$ by $\det(\hat{M})$. In addition, let $\overline{M}$ be the matrix obtained from 
$\hat{M}$ by setting $a_0=1$. Then we must show all partials in $a_1,\dots,a_n$ 
of $F(1,a_1,\dots,a_n)=\det(\overline{M})$ of order at most $m-1$ vanish at $(p_{11},\dots,p_{1n})$.
We will hereafter set  $F=F(1,a_1,\dots,a_n)$ and $P_1=(p_{11},\dots,p_{1n})$.
Note that $F(P_1)=0$ since $F$ is the determinant of a matrix two of whose rows are $w(P_1)$.

It will be convenient to keep track of the order of the partials occurring in various rows of $\overline{M}$.
Let $\lambda_j$ be the order of the partial occurring in row $j$ of $\overline{M}$ for relevant rows of
$\overline{M}$ (so we restrict $j$ to $1\leq j\leq \binom{d+n}{n}-(r+1)$).
Thus $\lambda_1=0$, $\lambda_j=1$ for $2\leq j\leq n+1$, etc.

Consider any partial $\partial_{i_1}$, $i_1>0$, applied to $F$. Using $F=\det(\overline{M})$ and Equation \eqref{multidet}, 
we see that $\partial_{i_1} F$ is a sum of determinants of the matrices obtained by
applying $\partial_{i_1}$ to successive rows of $\overline{M}$. The bottom $r+1$ rows of $\overline{M}$ do not involve
any of the variables $a_1,\dots,a_n$ and so applying $\partial_{i_1}$ to such a row 
gives a zero row, and hence a zero determinant.
For a row $j$ with $1\leq j\leq \binom{d+n}{n}-(r+1)$ for which $\lambda_j<m-1$, applying
$\partial_{i_1}$ to that row gives a row which already occurs further down in $\overline{M}$, and thus again
the determinant of that matrix is 0. Thus the only rows to which we apply $\partial_{i_1}$
for which the determinant might not be 0 are rows $j$ with $\lambda_j=m-1$. 

Now consider $\partial_{i_2} \partial_{i_1} F$, so we must apply $\partial_{i_2}$ to matrices to which we already
applied $\partial_{i_1}$. By the same argument, the only rows $j$ for which the determinant 
of one of these matrices might not be 0
are rows with $\lambda_j$ equal to $m-2$ or $m-1$. This is because applying $\partial_{i_2}$ to a row $j$ with $\lambda_j<m-2$,  
gives a result which already occurs as a row $j'>j$ with $\lambda_{j'}=\lambda_j+1\leq m-2$, and thus 
the resulting matrix has determinant 0. (If we apply $\partial_{i_2}$ to a row with $\lambda_j=m-1$,
the result does not occur further down the matrix because $m-1$ is the maximum order of the partials
in $\overline{M}$, and if we apply $\partial_{i_2}$ to a row with $\lambda_j=m-2$,
the result need not occur further down the matrix because the result we get is a partial of order $m-1$
and some of the partials of this order already in the matrix got changed when we applied $\partial_{i_1}$.)

Continuing in this way we see for $k\leq m-1$, if $i_l>0$ for all $l$, that $\partial_{i_k}\cdots\partial_{i_1} F$ 
is a sum of determinants of matrices obtained by
applying the partials $\partial_{i_l}$ to rows $j$ of $\overline{M}$ with $\lambda_j\geq m-k>0$.
Thus each matrix has $w(1,a_1,\dots,a_n)$ as its first row. Thus its determinant evaluated at $P_1$ is 0,
since $w(P_1)$ is also a row of the matrix, further down.
Thus all partials of $F$ of order at most $m-1$ in $a_1,\dots,a_n$ evaluated at $P_1$ vanish.
\end{proof}

\section{BMSS duality}

\begin{theorem}\label{tangent}
Let $d, m$ and $F(a_0,\dots,a_n,x_0,\dots,x_n)$ be as in Proposition \ref{pierwsze}, let $B=(b_0,\dots,b_n)\in\PP^n$ and define
homogeneous polynomials in $\mathbb{F}[\x0n]$ by setting $F_L := F(B,\x0n)$ and $F_R := F(\x0n,B)$.
If $F_L(\x0n) \not \equiv 0$,
then each partial derivative of $F_L$ of order $m$ at $B$ is equal to $(-1)^m$ times the same partial of $F_R$ at $B$.
\end{theorem}

\begin{corollary}\label{BMSSCor}
Let $d, m$, $F(a_0,\dots,a_n,x_0,\dots,x_n)$, $B$, $F_L(\x0n)$ and $F_L(\x0n)$ be as in Theorem \ref{tangent}.
If $F_L(\x0n)$ has multiplicity exactly $m$ at $B$, then 
$F_L$ and $F_R$ both have multiplicity $m$ at $B$ and the tangent cones of $F_L$ and $F_R$ are equal at $B$. 
\end{corollary}

\begin{proof}
For each $B$, $F_L(x_0,\dots,x_n)$ has multiplicity at least $m$ at $B$ (see Proposition \ref{pierwsze}).
By Theorem \ref{m-at-Z-and_Pv1}, $F_R(x_0,\dots,x_n)$ also has multiplicity at least $m$ at $B$.
If $F_L(\x0n)$ has multiplicity exactly $m$ at $B$, Theorem \ref{tangent} implies that $F_L$ and $F_R$ both
have multiplicity $m$ at $B$, and moreover that the tangent cones are equal at $B$. 
(The reason that the tangent cones coincide is because, given a homogeneous polynomial $H(\x0n)$,
the tangent cone for $H(\x0n)=0$ at a point $X$ of multiplicity $m$ is given by the equation
$$\sum_{j_0+\dots+j_n=m}\frac{1}{j_0!\cdots j_n!} \frac{\partial^m(H)}{\partial x_0^{j_0} 
\cdots \partial x_n^{j_n}}\left(X\right) x_0^{j_0} \cdots x_n^{j_n}=0.)$$
\end{proof}

\begin{remark}\label{remTC}
The fact that $F_L(\x0n)$ and $F_R(\x0n)$ have the same tangent cone at $B$ is a version of
what has come to be known as BMSS duality. The fact that there was a connection between $F_L(\x0n)$ and $F_R(\x0n)$
at $B$ was first remarked on by \cite{BMSS} in the case of the unexpected curve discussed in Example 
\ref{exb3r}. The paper \cite{H-M-N-T} showed that the connection was in terms of tangent cones
and that it held for many cases of unexpected plane curves and for at least some cases of unexpected hypersurfaces
in $\PP^n$, but \cite{H-M-N-T} 
established this fact by a fairly ad hoc proof.
Our result here is more general and depends on fundamental properties of determinants.
One thing that \cite{H-M-N-T} did that we do not do in our
more general setting is to determine the degree of the factor of $F(a_0,\dots,a_n,x_0,\dots,x_n)$
involving just the variables $a_i$. Our Theorem \ref{k-mult} does however partially address this;
see Example \ref{exb3r}.
\end{remark}

\begin{proof}[Proof of Theorem \ref{tangent}]
Recall that $F=\det(M)$ where

{\tiny	$$
	M=\left[\begin{array}{cccccc}
		\frac{\partial^{m-1}}{\partial a_0^{m-1} }(a_0^d)&	\frac{\partial^{m-1}}{\partial a_0^{m-1} }(a_0^{d-1}a_1) &	\frac{\partial^{m-1}}{\partial a_0^{m-1} }(a_0^{d-1}a_2) &
		\dots&
		\frac{\partial^{m-1}}{\partial a_0^{m-1} }(a_{n-1}a_n^{d-1}) &
		\frac{\partial^{m-1}}{\partial a_0^{m-1} }(a_n^d)  \\
		
		\frac{\partial^{m-1}}{\partial a_0^{m-2}\partial a_1}(a_0^d)&	\frac{\partial^{m-1}}{\partial a_0^{m-2}\partial a_1 }(a_0^{d-1}a_1) &	\frac{\partial^{m-1}}{\partial a_0^{m-2}\partial a_1 }(a_0^{d-1}a_2) &
		\dots&
		\frac{\partial^{m-1}}{\partial a_0^{m-2}\partial a_1 }(a_{n-1}a_n^{d-1}) &
		\frac{\partial^{m-1}}{\partial a_0^{m-2}\partial a_1 }(a_n^d) \\
		\vdots&\vdots&\vdots& &\vdots&\vdots\\
		\frac{\partial^{m-1}}{\partial a_n^{m-1}}(a_0^d)&	\frac{\partial^{m-1}}{\partial a_n^{m-1} }(a_0^{d-1}a_1) &	\frac{\partial^{m-1}}{\partial a_n^{m-1} }(a_0^{d-1}a_2) &
		\dots&
		\frac{\partial^{m-1}}{\partial a_n^{m-1} }(a_{n-1}a_n^{d-1}) &
		\frac{\partial^{m-1}}{\partial a_n^{m-1}}(a_n^d)
		\\
		p_{10}^d& p_{10}^{d-1}p_{11} & p_{10}^{d-1}p_{12}&\ldots&p_{1n-1}p_{1n}^{d-1}&p_{1n}^{d}\\
		\vdots&\vdots&\vdots& &\vdots&\vdots\\
		p_{r0}^d& p_{r0}^{d-1}p_{r1} & p_{r0}^{d-1}p_{r2}&\ldots&p_{rn-1}p_{rn}^{d-1}&p_{rn}^{d}\\
		x_{0}^d& x_0^{d-1}x_1 & x_0^{d-1}x_2&\ldots&x_{n-1}x_n^{d-1}&x_n^{d}
	\end{array}\right].
	$$}
 
We can compute $\det(M)$ using successive Laplace expansion along the rows
whose entries come from evaluating at the points $P_i$ (hence rows 2 through $r+1$, counting up
from the bottom). This reduces computing $\det(M)$ to computing 
determinants of $J\times J$ submatrices $S$ (where $J=\binom{n+m-1}{n}-r$), as below (and summing after multiplying 
them by a constant coming from entries in rows 2 through $r+1$, from the bottom):
 
 {\tiny	$$S=
 	\left[\begin{array}{cccc}
 		\frac{\partial^{m-1}}{\partial a_0^{m-1}}(a_{0}^{\alpha_{1,0}}a_1^{\alpha_{1,1}}\cdots a_n^{\alpha_{1,n}})&	\frac{\partial^{m-1}}{\partial a_0^{m-1} }(a_{0}^{\alpha_{2,0}}a_1^{\alpha_{2,1}}\cdots a_n^{\alpha_{2,n}}) &
 	\dots&\frac{\partial^{m-1}}{\partial a_0^{m-1}}(a_{0}^{\alpha_{n,0}}a_1^{\alpha_{n,1}}\cdots a_n^{\alpha_{n,n}})
 	 \\ 	
 	\frac{\partial^{m-1}}{\partial a_0^{m-2}\partial a_1}(a_{0}^{\alpha_{1,0}}a_1^{\alpha_{1,1}}\cdots a_n^{\alpha_{1,n}})&	\frac{\partial^{m-1}}{\partial a_0^{m-2}\partial a_1}(a_{0}^{\alpha_{2,0}}a_1^{\alpha_{2,1}}\cdots a_n^{\alpha_{2,n}}) &	
 	\ldots&\frac{\partial^{m-1}}{\partial a_0^{m-2}\partial a_1}(a_{0}^{\alpha_{n,0}}a_1^{\alpha_{n,1}}\cdots a_n^{\alpha_{n,n}}
  \\
 	\vdots &\vdots&\vdots&\vdots
 	\\
 		\frac{\partial^{m-1}}{\partial a_n^{m-1}}(a_{0}^{\alpha_{1,0}}a_1^{\alpha_{1,1}}\cdots a_n^{\alpha_{1,n}})&	\frac{\partial^{m-1}}{\partial a_n^{m-1}}(a_{0}^{\alpha_{2,0}}a_1^{\alpha_{2,1}}\cdots a_n^{\alpha_{2,n}})&
 		\ldots&
 			\frac{\partial^{m-1}}{\partial a_n^{m-1}}(a_{0}^{\alpha_{n,0}}a_1^{\alpha_{n,1}}\cdots a_n^{\alpha_{n,n}})
 	\\
 	x_{0}^{\alpha_{1,0}}x_1^{\alpha_{1,1}}\cdots x_n^{\alpha_{1,n}}& x_{0}^{\alpha_{2,0}}x_1^{\alpha_{2,1}}\cdots x_n^{\alpha_{2,n}} & \ldots &	x_{0}^{\alpha_{n,0}}x_1^{\alpha_{n,1}}\cdots x_n^{\alpha_{n,n}}
 	\end{array}\right].
 	$$}

\noindent So in the last row we have $J$ distinct monomials in $\x0n$ of degree $d$
and in the rows above we have all partials of order $m-1$ for the corresponding monomials in the variables $\a0n$. 
Thus the entries in the first $\binom{n+m-1}{n}$ rows of $S$ are of the form
$$\frac{\partial^{m-1}}{\partial a_0^{t_0}\cdots\partial a_n^{t_n}}(a_{0}^{\alpha_{k,0}}a_1^{\alpha_{k,1}}\cdots a_n^{\alpha_{k,n}})=
\Big(\Pi_{i=0}^n\binom{\alpha_{k,i}}{t_i}t_i!\Big)a_0^{\alpha_{k,0}-t_0}a_1^{\alpha_{k,1}-t_1}\cdots a_n^{\alpha_{k,n}-t_n}$$
where $t_0+\cdots+t_n=m-1$. This specific expression is the entry of the matrix $S$ in the $k$th column 
of the row corresponding to the partial $\frac{\partial^{m-1}}{\partial a_0^{t_0}\cdots\partial a_n^{t_n}}$.

Thus it is enough for $t_0+\dots+t_n=m$ to show that 
$$(-1)^m\Big(\frac{\partial^{m-1}}{\partial a_0^{t_0}\cdots\partial a_n^{t_n}}(\det S)\Big)(b_0,\dots,b_n,b_0,\dots,b_n)=$$
$$\Big(\frac{\partial^{m-1}}{\partial x_0^{t_0}\cdots\partial x_n^{t_n}}(\det S)\Big)(b_0,\dots,b_n,b_0,\dots,b_n).$$

This will take several steps. Expanding $\det S$ along the last row, we get 
$$\det S=D(\x0n,\a0n)= \sum_{k=1}^Jx_0^{\alpha_{k,0}}x_1^{\alpha_{k,1}}\cdots x_n^{\alpha_{k,n}} D_k(\a0n),$$
where $D_k(\a0n)$ is the cofactor of the entry $x_{0}^{\alpha_{k,0}}x_1^{\alpha_{k,1}}\cdots x_n^{\alpha_{k,n}}$ in $S$.

To compute $D_k$ we define a multigrading on $\mathbb{F}[\x0n,\a0n,\gamma_0,\dots,\gamma_n]$ by $\deg(\gamma_j)=\deg(x_j)=\deg(a_j)=(0,\dots,0,1,0,\dots,0)$, where the $1$ is in the $(j+1)$st position. 
We can now take $\widetilde{S}$ to be the matrix obtained from $S$ by using $ \gamma_0,\dots,\gamma_n$ to
homogenize the entries of $S$ involving partials. For example, 
$$\frac{\partial^{m-1}}{\partial a_0^{t_0}\cdots\partial a_n^{t_n}}(a_{0}^{\alpha_{k,0}}a_1^{\alpha_{k,1}}\cdots a_n^{\alpha_{k,n}})$$
in $S$ changes in $\widetilde{S}$ to
$$\frac{\partial^{m-1}}{\partial a_0^{t_0}\cdots\partial a_n^{t_n}}(	a_{0}^{\alpha_{k,0}}a_1^{\alpha_{k,1}}\cdots a_n^{\alpha_{k,n}})\gamma_0^{t_0}\cdots\gamma_n^{t_n}.$$

Observe that each nonconstant entry in column $k$ in $\widetilde{S}$ is 
multihomogeneous of multidegree $(\alpha_{k,0},\dots,\alpha_{k,n})$, 
hence $\det \widetilde{S}$ is multihomogeneous of multidegree 
$(\widetilde{A_1},\widetilde{A_2},\dots,\widetilde{A_n})$ where $\widetilde{A_c}=\sum_j\alpha_{j,c}$. 
We also see that we obtain $\widetilde{S}$ from $S$ by multiplying each row of $S$ 
corresponding to a partial $\frac{\partial^{m-1}}{\partial a_0^{t_0}\cdots\partial a_n^{t_n}}$
by the monomial $\gamma_0^{t_0}\cdots\gamma_n^{t_n}$.
Thus $\det \widetilde{S}$ is obtained by multiplying $\det S$ by the product
of all monomials $\gamma_0^{t_0}\cdots\gamma_n^{t_n}$ of degree $m-1$,
which is just $(\gamma_0\cdots\gamma_n)^t$, where $t=\binom{n+m-1}{n}\frac{m-1}{n+1}$.
Therefore $\det S$ is multihomogeneous of multidegree $(A_1,\dots,A_n)$, where 
 $A_c=\widetilde{A_c}-t$ for $c=0,\dots,n$.
 Consequently 
 $$D_k(\a0n)=C_ka_0^{A_0-\alpha_{k,0}}\cdots a_n^{A_n-\alpha_{k,n}}.$$
 
We will use the following Lemma:
 
\begin{lemma}\label{ck}
Given the coefficients $C_k$ and exponents $\alpha_{k,i}$ as above, we have
\begin{equation}\label{numera}
\sum_{k=1}^JC_k\alpha_{k,0}^{t_0}\alpha_{k,1}^{t_1}\cdots\alpha_{k,n}^{t_n}=0,
\end{equation}
for every $t_0+\dots+t_n\leq m-1$.
\end{lemma}
 
\begin{proof}[Proof of the lemma]
Observe that each partial of order $m-1$ (with respect to the variables $\x0n$) of $\det S$ evaluates to 0 at the point $(\a0n)$, 
since in the matrix for $S$ we differentiate the last row, which after substituting $\a0n$ in for $\x0n$ 
becomes equal to a row higher up in $S$, so the determinant is zero. Moreover, if we compute a partial of $\det S$ of order 
$k<m-1$ and evaluate at $(\a0n)$, the last row will (by Euler's identity) be
a linear combination of the rows higher up in $S$ corresponding to the partials of order $m-1$, 
hence the matrix will not have maximal rank so again
will have determinant 0. 

Thus from $\det S=\sum_{k=1}^Jx_0^{\alpha_{k,0}}x_1^{\alpha_{k,1}}\cdots x_n^{\alpha_{k,n}} D_k(\a0n)$,
using $h=t_0+\dots+t_n$ for $0\leq h\leq m-1$, we have
$$0=\Big(\frac{\partial^{h}}{\partial x_0^{t_0}\ldots\partial x_n^{t_n}}(\det S)\Big)(\a0n,\a0n)=$$
$$\sum_{k=1}^J\Big(\Pi_{i=0}^n\binom{\alpha_{k,i}}{t_i}t_i!\Big)a_0^{\alpha_{k,0}-t_0}a_1^{\alpha_{k,1}-t_1}\cdots a_n^{\alpha_{k,n}-t_n}
D_k(a_0,\dots,a_n)=$$
$$\sum_{k=1}^J\Big(\Pi_{i=0}^n\binom{\alpha_{k,i}}{t_i}t_i!\Big)a_0^{\alpha_{k,0}-t_0}a_1^{\alpha_{k,1}-t_1}\cdots a_n^{\alpha_{k,n}-t_n}
C_ka_0^{A_0-\alpha_{k,0}}\cdots a_n^{A_n-\alpha_{k,n}}=$$
$$\sum_{k=1}^JC_k\Big(\Pi_{i=0}^n\binom{\alpha_{k,i}}{t_i}t_i!\Big)a_0^{A_0-t_0}a_1^{A_1-t_1}\cdots a_n^{A_n-t_n},$$
where we interpret $\binom{\alpha_{k,i}}{t_i}$ to be 0 if $t_i>\alpha_{k,i}$.
Factoring out $a_0^{A_0-t_0}a_1^{A_1-t_1}\cdots a_n^{A_n-t_n}$ gives
$$\sum_{k=1}^JC_k\Big(\Pi_{i=0}^n\binom{\alpha_{k,i}}{t_i}t_i!\Big)=0.$$

Taking $h=0$, i.e.  $t_0=t_1=\cdots=t_n=0$, we get
$$\sum_{k=1}^JC_k=0.$$
Taking $h=1$ and $t_i=0$ except for $t_j=1$, we get 
$$\sum_{k=1}^JC_k\alpha_{k,j}=0.$$
Taking $h=2$ and $t_i=0$ except for $t_i=t_j=1$ for $i\neq j$, we get 
$$\sum_{k=1}^JC_k\alpha_{k,i}\alpha_{k,j}=0$$
while for $t_i=0$ except for $t_j=2$, we get
$$\sum_{k=1}^JC_k\alpha_{k,j}(\alpha_{k,j}-1)=0,$$
and so on.

Let us proceed by induction. We have (\ref{numera}) for $t_0+\dots+t_n=0$.
Suppose we have it for all $s_0+\dots+s_n=\ell-1< m-1$
and we wish to show it for some $t_0+\dots+t_n=\ell$.
Expanding $\sum_{k=1}^JC_k\Big(\Pi_{i=0}^n\binom{\alpha_{k,i}}{t_i}t_i!\Big)=0$ and
multiplying out gives (for appropriate indexed coefficients $b$ and $\beta$)
$$0=\sum_kC_k\big(\alpha_{k,0}(\alpha_{k,0}-1)\cdots(\alpha_{k,0}-(t_0-1))\big)\cdots\big(\alpha_{k,n}(\alpha_{k,n}-1)\cdots(\alpha_{k,n}-(t_n-1))\big)=$$
$$\sum_kC_k\big(\alpha_{k,0}^{t_0}+b_{1,0}\alpha_{k,0}^{t_0-1}+\dots+b_{t_0,0}\alpha_{k,0}\big)
\cdots\big(\alpha_{k,n}^{t_n}+b_{1,n}\alpha_{k,n}^{t_n-1}+\dots+b_{t_n,n}\alpha_{k,n}\big)=$$
$$\sum_kC_k\alpha_{k,0}^{t_0}\cdots\alpha_{k,n}^{t_n}+
\sum_{1\leq s_0+\dots+s_n<\ell}\beta_{0,\dots,n}\sum_kC_k\alpha_{k,0}^{s_0}\cdots\alpha_{k,n}^{s_n}=
\sum_kC_k\alpha_{k,0}^{t_0}\cdots\alpha_{k,n}^{t_n}.$$
\end{proof}

We now return to the proof of Theorem \ref{tangent}.
For $t_0+\dots+t_n=m$ we get 
$$\Big(\frac{\partial^{h}}{\partial x_0^{t_0}\ldots\partial x_n^{t_n}}(\det S)\Big)(b_0,\dots,b_n,b_0,\dots,b_n)=$$
$$\sum_{k=1}^JC_k\Big(\Pi_{i=0}^n\binom{\alpha_{k,i}}{t_i}t_i!\Big)b_0^{A_0-t_0}b_1^{A_1-t_1}\cdots b_n^{A_n-t_n}=$$
$$ \sum_{k=1}^JC_k\Pi_{i=1}^n\big(\alpha_{k,i}(\alpha_{k,i}-1)\cdots (\alpha_{k,i}-(t_i-1))\big)b_0^{A_0-t_0}b_1^{A_1-t_1}\cdots b_n^{A_n-t_n}=$$
$$\sum_{k=1}^JC_k\alpha_{k,0}^{t_0}\alpha_{k,1}^{t_1}\cdots \alpha_{k,n}^{t_n}b_0^{A_0-t_0}b_1^{A_1-t_1}\cdots b_n^{A_n-t_n},$$
where the last equality comes from 
multiplying out (as at the end of the proof of Lemma \ref{ck}) and using Lemma \ref{ck} to eliminate sums  
$\sum_kC_k\alpha_{k,0}^{s_0}\cdots \alpha_{k,n}^{s_n}$ with $s_0+\cdots+s_n < m$.

On the other hand, using
$$\det S= \sum_{k=1}^JC_kx_0^{\alpha_{k,0}}x_1^{\alpha_{k,1}}\cdots x_n^{\alpha_{k,n}} a_0^{A_0-\alpha_{k,0}}\cdots a_n^{A_n-\alpha_{k,n}}$$
we also have, in a similar way,
$$\Big(\frac{\partial^{h}}{\partial a_0^{t_0}\ldots\partial a_n^{t_n}}(\det S)\Big)(b_0,\dots,b_n,b_0,\dots,b_n)=$$
$$ \sum_{k=1}^JC_k\Pi_{i=1}^n\big((A_i-\alpha_{k,i})(A_i-\alpha_{k,i}-1)\cdots(A_i-\alpha_{k,i}-t_i+1)\big)b_1^{A_1-t_1}\cdots b_n^{A_n-t_n}=$$
$$(-1)^m\sum_{k=1}^JC_k\alpha_{k,0}^{t_0}\alpha_{k,1}^{t_1}\cdots \alpha_{k,n}^{t_n}b_1^{A_1-t_1}\cdots b_n^{A_n-t_n}.$$
\end{proof}
 
 \section{A refinement of Theorem \ref{m-at-Z-and_Pv1}}
We now present a refinement of Theorem \ref{m-at-Z-and_Pv1}
which allows us to understand factors of $F$. As  
the example from the introduction shows (see Example \ref{exb3r}),
$F(a_0,\dots,a_n,x_0,\dots,x_n)$ need not have factors involving only the variables $x_i$ but it
can have factors that involve only the variables $a_i$. 
Thinking of $F$ as defining a hypersurface in $\PP^n\times\PP^n$,
a factor that involves only the variables $a_i$
defines a locus of points $B$ such that $F\in I(B\times\PP^n)$.
If the factor occurs with multiplicity $k$, then $F\in I(B\times\PP^n)^k$.
Thus to understand the occurrence of such factors, it is useful to understand
under what circumstances there are points $B$ such that
$F\in I(B\times\PP^n)^k$. This is what we do in Theorem \ref{k-mult}.
As an added benefit, we can also use Theorem \ref{k-mult}
to get a short proof of Theorem \ref{m-at-Z-and_Pv1}.

To set up the statement of the theorem,
let $F=\det M$, where $M$ is as in Proposition \ref{pierwsze}, and let $L_j$  be the 
vector space of forms $[I(jB)\cap I(Z)]_d$ of degree $d$ vanishing at a point $B=(b_0, b_1, \dots, b_n)\in\PP^n$ with multiplicity $j$ 
and vanishing on $Z=P_1+\dots+P_r$ for distinct points $P_i$. We define $k_j=\max(0,\dim L_j-\binom{n+j-1}{n}-r)$;
thus when $B\not\in Z$, $k_j$ is the superabundance of the space $L_j$
(meaning $k_j$ is the actual dimension of $L_j$ minus the expected dimension of $L_j$).
Finally we recall that the ideal $I(B)$ of the point $B\in\PP^n$ is generated by 
the $2\times 2$ minors of the matrix
$$\left[\begin{array}{cccc}
b_0& b_1& \ldots& b_n\\
a_0 &a_1 &\ldots&a_n
\end{array}\right].$$
In $\mathbb{F}[\a0n,\x0n]$ the $2\times 2$ minors generate the ideal of $B\times \PP^n$.

\begin{theorem}\label{k-mult}
In the context of the preceding paragraph we have
	$$F\in I(B\times \PP^n)^k\subset \mathbb{F}[\a0n,\x0n]$$
where $k:=\sum_{j=1}^mk_j$.
\end{theorem}

\begin{proof} 
As in the proof of Theorem \ref{m-at-Z-and_Pv1}, it is enough to show that the partials
of $F=\det(\overline{M})$ of order at most $k-1$ in the variables $a_i$, $i>0$, vanish at $B$,
where we assume that $b_0\neq 0$.
For this proof we obtain a matrix $K$ from
$\overline{M}$ by removing the last row of $\overline{M}$, so
$$K=\left[\begin{array}{c}
        w(1,a_1,\dots,a_n) \\
\frac{\partial w(1,a_1,\dots,a_n)}{\partial a_1 } \\
	\frac{\partial w(1,a_1,\dots,a_n)}{\partial a_2 } \\
	\vdots\\
	\frac{\partial w(1,a_1,\dots,a_n)}{\partial a_n }\\
	\frac{\partial^2w(1,a_1,\dots,a_n)}{\partial a_1^2 } \\
	\vdots\\
	\frac{\partial^{m-1}w(1,a_1,\dots,a_n)}{\partial a_n^{m-1}} \\
        w(P_1) \\
        \vdots \\
        w(P_r) \\
	\end{array}\right].
	$$
Observe that $\det\overline{M}=0$ at $(B,\x0n)$ if and only if $\overline{M}$ has 
deficient rank at $(B,\x0n)$ if and only if $K$ has deficient rank at $B$. 
	
We will argue contrapositively, so assume for some $t$ that
\begin{equation}\label{tw7a}
F\notin  I(B\times \PP^n)^t.
\end{equation}
Obviously 
\begin{equation}\label{tw7b}
F\not\equiv 0.
\end{equation}
We want to show that $k<t$.

By (\ref{tw7a}) some partial derivative of $F$ of order $\tau\leq t-1$ evaluated at $B$ is non-zero. Thus expanding
the derivative of $\det(\overline{M})$ as in \eqref{multidet} we must get at least one nonzero summand.
This summand is a determinant of a matrix (necessarily of maximal rank), 
which is obtained from $\overline{M}$ by taking $\tau$  derivatives of rows and evaluating at $B$.
Of course applying a derivative to any of the last $r+1$ rows of $\overline{M}$ 
(or equivalently to any of the last $r$ rows in $K$) gives a matrix with deficient rank, so we may assume all the 
derivatives are applied to rows above the last $r$ rows of $K$.
 
 So let us assume that after applying $\tau$ partial derivatives to rows of $K$ 
 (the possibility that more than one derivative is applied to one row is not excluded) and evaluating at $B$ we get
 \begin{equation}\label{tw7c}
 \textrm {a matrix } \widetilde{K}_B \textrm{ with maximal rank.}
 \end{equation}

Let $s_j$ be the number of rows of $K_j$; since $K_m=K$, we have $s_m=\binom{d+n}{n}-1$
and we define $s=s_m$.
Let $d_i^j$ denote how many partials are involved in defining row $i$ of $K_j$
(so the $j$ in $d_i^j$ is an index, not an exponent).
For example, $d_1^j=0$ for all $j$ and also $d_i^j=0$ for $i>s_j-r$ for all $j$, 
$d_i^j=1$ for $2\leq i\leq n+2$ for $j>1$, etc.

Let $u_i$ denote how many of the $\tau$ partials get applied to the $i$th row of $K$. 
Before proceeding observe that our definition of $k_i$ says that the rank of the matrix 
$$
K_1=\left[\begin{array}{c}
w\\
w(P_1)\\
\vdots\\
w(P_r)
\end{array}\right]
$$
evaluated at $B$ is less by $k_1$ than the maximum possible, 
the rank of 
$$
K_2=\left[\begin{array}{c}
w\\
\frac{\partial w}{\partial a_1 } \\
\vdots\\
\frac{\partial w}{\partial a_n }\\
w(P_1)\\
\vdots\\
w(P_r)
\end{array}\right]
$$
evaluated at $B$ is less by $k_2$ than the maximum possible, and so on.  
Note that $K_j$ consists of the rows of $K$ involving partials of order at most $j-1$.

We want to show that $k<\tau$. Thus we must relate $k$ to how many partials we apply to $K$.
But $k=k_1+\dots+k_m$, so our proof will be to relate each $k_j$ to how many partials are applied to
each submatrix $K_j$ of $K$. 
To do that, we will use recursively computed quantities $u_i^j$, which are directly related to the
$k_j$ and give lower bounds for the number of partials applied to row $i$ of $K_j$.
(Thus, as for $d_i^j$, the $j$ in $u_i^j$ is an index and not an exponent.)

{\bf Claim.}
Given $j$, $1\leq j\leq m$, there exists a sequence $u^j=(u_1^j,\dots,u_s^j)$ such that: 
\begin{itemize}
	\item[(i)] $u_i^j=0$ for $i>s_j$;\\
	\item[(ii)] $0\leq u_i^j\leq u_i$;\\
	\item[(iii)] $\sum_{i=1}^s u_i^j =k_1+\dots+k_j$;\\
	\item[(iv)] $u_i^j+d_i^j\leq j$.
\end{itemize}

Suppose we have the claim. Then $k=k_1+\dots+k_m=\sum_i u_i^m\leq \sum_i u_i = \tau\leq t-1$,
as we wanted to show.

We now prove the claim. Claim (i) will hold because $u_i^j$ will be concerned only with rows of $K_j$. 
For the other claims, we proceed by induction. Let $j=1$. If  $k_1\geq 2$ 
 then the rank of $K_1(B)$ is less than maximal by at least $2$, hence the rank of
 
 $$
 \left[\begin{array}{c}
 w(P_1)\\
 \vdots\\
 w(P_r)
 \end{array}\right]
 $$
 is less than maximal, so also the rank of $K$ is less than maximal, contradicting (\ref*{tw7b}).

Thus $k_1=0$ or $1$. If $k_1=0$, take $u^1=(0,\dots,0)$, and this sequence satisfies the Claim for $j=1$.
If $k_1=1$ define  $u^1=(1,0,\dots,0)$.  To prove that $u^1$ satisfies the Claim for $j=1$, observe that 
$u_1\geq 1$, as otherwise none of the partial derivatives are applied to the row $w$, hence it survives 
unchanged and, after evaluating at $B$, the matrix $K_1(B)$ would be a submatrix of 
$\widetilde{K}_B$, contradicting (\ref{tw7c}). 

Suppose now that we have the required $u^{j-1}$ as in the claim; we want to find $u^j$.
Consider the matrix $K_j$. Apply $u_1^{j-1}$ partial derivatives to the first row of
$K_j$ obtaining $K_j'$. Since $u_1^{j-1}=u_1^{j-1}+d_1^{j-1}\leq j-1$ (by Claim (iv)
and the fact that $d_1^{j-1}=0$), we 
get (in the first row of $K_j'$) a partial derivative of $w$ of order at most $j-1$
which thus occurs in some other row of $K_j$ further down, so the rank of $K_j'(B)$ is at most the rank of $K_j(B)$. 
 
We proceed with $u_2^{j-1}$ 
partial derivatives applied to the second row of $K_j'$ and so on. After 
applying all partial derivatives from $u^{j-1}$ we have obtained a matrix $\widetilde{K_j}$, which by the argument
above has
$$\rank \widetilde{K_j}(B)\leq \rank K_j(B).$$
Since $K_j(B)$ is short of maximal rank by $k_j$, so is $\widetilde{K_j}(B)$.
Thus we need to change at least $k_j$ rows of $\widetilde{K_j}(B)$ to bring it to maximal rank.
After applying additional derivatives (if need be) to the rows of $K_j(B)$, coming from the fact
that applying all derivatives described by $u^{j-1}$  
may only be a part of applying all the derivatives described by $u$,
we get a matrix $\widetilde{K_j}(B)'$ which is a submatrix of $\widetilde{K}_B$,
which must therefore have maximal rank. This means that $u$ and $u^{j-1}$
differ in at least $k_j$ positions corresponding to $k_j$ rows of $\widetilde{K_j}(B)$.

Thus at least $k_j$ values among $u_1-u_1^{j-1},\dots,u_{s_j-r}-u_{s_j-r}^{j-1}$ are positive. 
We choose exactly $k_j$ such indices and define   
$$u_i^j=
 \begin{cases} u_i^{j-1}+1& \text{ if $i$ among chosen ones}\\
                     u_i^{j-1}& \text{ otherwise.}
\end{cases}$$
Observe that obviously $0\leq u_i^j\leq u_i$ with $u_i^j=0$ for $i>s_j$, 
and $\sum_i u_i^j=\sum_i u_i^{j-1}+k_j=k_1+\dots +k_j$ by induction.
For $i\leq s_{j-1}-r$ we have $d_i^j=d_i^{j-1}$ so
$u_i^j+d_i^j\leq 1+u_i^{j-1}+ d_i^{j-1}\leq 1+(j-1)=j$, while
for $s_{j-1}-r<i\leq s_j-r$ we have $d_i^j=j-1$ so
$u_i^j+d_i^j\leq 1+(j-1)=j$.
\end{proof}

We now use Theorem \ref{k-mult} to give an alternate proof of Theorem \ref{m-at-Z-and_Pv1}.

\begin{proof}[Proof of Theorem \ref{m-at-Z-and_Pv1}]
The degree of $F(a_0,\dots,a_n,S)$ follows from
Proposition \ref{pierwsze}.
Now let $B=P_1$ (the proof for other $P_i$ is the same) and let $Z=P_1+\cdots+P_r$. So we are interested in
the vector spaces $L_j=[I(jP_1)\cap I(Z)]_d$. Since $P_1$ occurs in $Z$,
we have $[I(jP_1)\cap I(Z)]_d=[I(jP_1)\cap I(P_2+\cdots+P_r)]_d$,
so $\dim [I(jP_1)\cap I(Z)]_d=\dim [I(jP_1)\cap I(P_2+\cdots+P_r)]_d
\geq \binom{d+n}{n}-\binom{n+m-1}{n}-(r-1)>\binom{d+n}{n}-\binom{n+m-1}{n}-r$, hence
so $k_j\geq 1$ for all $j$. Thus $k=k_1+\dots+k_m\geq m$, so 
$F\in I(P_1\times\PP^n)^m$, so $F(a_0,\dots,a_n,S)$ vanishes to order $m$ at $P_1$.
Because of symmetry in the matrix $M$ with $\det M=F$, where $S$ plays the same role
as $P_1$ (both giving rise to a row of $M$), we also have that  
$F(a_0,\dots,a_n,S)$ vanishes to order $m$ at $S$.
\end{proof}
 
\begin{example}\label{exb3r}
Consider Theorem \ref{tangent} in the case of the example from the introduction,
so $d=4$, $m=3$ and $Z = \{ P_1,\dots,P_8 \}$. As in the introduction,
we will use variables $a,b,c$ instead of $a_0,a_1,a_2$, and similarly $x,y,z$ for 
$x_0,x_1,x_2$. For some choice of of the ordering $w$ of the monomials of degree 4
in $a,b,c$, the matrix $M$ has the form
{\tiny $$M=\left[\begin{array}{ccccccccccccccc}
12a^2&0&   0&   6ab&6ac&0&  0&  0&  0&  2bc& 0&   0&   2b^2& 2c^2& 0  \\
0&   0&   0&   3a^2&0&  3b^2&0&  0&  0&  2ac& 2bc& c^2&  4ab& 0&   0  \\
0&   0&   0&   0&  3a^2&0&  0&  3c^2&0&  2ab& b^2&  2bc& 0&   4ac& 0  \\
0&   12b^2&0&   0&  0&  6ab&6bc&0&  0&  0&   2ac& 0&   2a^2& 0&   2c^2\\
0&   0&   0&   0&  0&  0&  3b^2&0&  3c^2&a^2&  2ab& 2ac& 0&   0&   4bc\\
0&   0&   12c^2&0&  0&  0&  0&  6ac&6bc&0&   0&   2ab& 0&   2a^2& 2b^2\\
1&   0&   0&   0&  0&  0&  0&  0&  0&  0&   0&   0&   0&   0&   0  \\
0&   1&   0&   0&  0&  0&  0&  0&  0&  0&   0&   0&   0&   0&   0  \\
0&   0&   1&   0&  0&  0&  0&  0&  0&  0&   0&   0&   0&   0&   0  \\
1&   1&   0&   1&  0&  1&  0&  0&  0&  0&   0&   0&   1&   0&   0  \\
1&   1&   0&   -1& 0&  -1& 0&  0&  0&  0&   0&   0&   1&   0&   0  \\
1&   0&   1&   0&  1&  0&  0&  1&  0&  0&   0&   0&   0&   1&   0  \\
1&   0&   1&   0&  -1& 0&  0&  -1& 0&  0&   0&   0&   0&   1&   0  \\
0&   1&   1&   0&  0&  0&  1&  0&  1&  0&   0&   0&   0&   0&   1  \\
x^4&  y^4&  z^4&  x^3y&x^3z&xy^3&y^3z&xz^3&yz^3&x^2yz&xy^2z&xyz^2&x^2y^2&x^2z^2&y^2z^2
\end{array}\right].$$}

Let $F=\det(M)$. Then (ignoring a constant factor of $-1728$) we have
\begin{multline*}
F=ab^3c^3(a+b-c)(a-b+c)(c^3x^3y-c^3xy^3-b^3x^3z+3ab^2x^2yz-3ac^2x^2yz\\
-3a^2bxy^2z+3bc^2xy^2z+a^3y^3z+3a^2cxyz^2-3b^2cxyz^2+b^3xz^3-a^3yz^3).
\end{multline*}

Theorem \ref{tangent} says at a point $B$ such that $F_L\not\equiv0$ that at $B$ the tangent cones for $F_L$ and $F_R$.
are the same. The hypothesis $F_L\not\equiv0$ is needed as we see from $F$ above.
If $B=(b_0,b_1,b_3)$ with $b_0=0$, then $F_L\equiv0$, but $F_R\not\equiv0$. 
Away from where the factor $ab^3c^3(a+b-c)(a-b+c)$ vanishes we have $F_L\not\equiv0$ and hence
the tangent cones will agree. We note that the linear factors of $ab^3c^3(a+b-c)(a-b+c)$
define lines in Figure \ref{Fig1}: $a=0$ is the dashed line through $P_9$,
$b=0$ and $c=0$ are the dashed lines through $P_1$, and $a+b-c=0$ and $a-b+c)=0$ are the 
solid lines through $P_8$.

We now explain how Theorem \ref{k-mult} helps to see, 
without computing the determinant, how factors like $ab^3c^3(a+b-c)(a-b+c)$ arise.
In fact, it is easy to find factors of $F$ not involving $x,y,z$. For example, the line $a=0$ contains
three points, $P_1$, $P_2$ and $P_8$. Take any point $P$ on this line, but not equal to any of the $P_j$.
Since a line $L$ containing three points and an additional point of multiplicity 3 must be
a component of the quartic, the dimension of the homogeneous component $[I(3P+Z)]_4$ of degree 4 of the ideal $I(3P+Z)$
is equal to the dimension of $[ I(2P+Z')]_3$, where $Z'$ consists of the 5 points of $Z$ not on $L$.
But this dimension is at least 2, so $k_3 \geq 1$, and $F$ must vanish along $\{a=0\} \times \PP^2$.
Similarly we get the multiplicity at least 3 of a factor $b$ and $c$ in $F$. Each of these lines contains 4 points, and computing $k_3$ as before gives $k_3 \geq 2$, computing $k_2$ gives $k_2 \geq 1$ for each point $P$ on any of these lines,
so $F$ vanishes to order at least $k_3+k_2=3$ on $\{b=0\} \times \PP^2$ and on $\{c=0\} \times \PP^2$.

Theorem \ref{m-at-Z-and_Pv1} states that for $S$ (such that $F(a,b,c,S) \not\equiv 0$), the curve given by $F_S(a,b,c) = F(a,b,c,S)$
has multiplicity at least $3$ at $P_1,\dots,P_8$ and $S$. In fact, we can see that this multiplicity is at least $4$ at each $P_j$, given by the constant (i.e. not depending on $S$) factors of $F_S$. At $P_3$ the multiplicity is, in fact, $6$. Observe also that comparing degrees of
components of $F_S$ we get immediately that the degrees must be 1 along lines with three points and 3 along lines with four points --- otherwise we cannot get a triple point at (sufficiently general) $S$.

Factor $ab^3c^3(a+b-c)(a-b+c)$ out of $F_S$ and denote the result by $G_S$. Thus $G_S$ is a polynomial of
degree 3. If $S$ is general enough (not lying on any of the lines defined by the linear factors of 
$ab^3c^3(a+b-c)(a-b+c)$), then $G_S$ must have a triple point at $S$, since (by 
Theorem \ref{m-at-Z-and_Pv1}) $F_S$ has a triple point at $S$. But then $G_S$ splits into three lines, which are tangent lines (at $S$) to a unique quartic passing through $3S$ and $Z$, by Theorem \ref{tangent}. 

We may ask about points $S$ such that $F_S \equiv 0$.
These points are the locus defined by the ideal in $\mathbb{C}[x,y,z]$ 
generated by the coefficients of $F$ regarded as being in $F \in \mathbb{C}[x,y,z][a,b,c]$.
This ideal describes 9 points, obviously including $P_1,\dots,P_8$, but also $P_9$, 
since the example comes from an unexpected curve through the 9 points. In our setting, each expected quartic in 
$3B+P_1+\dots+P_8$ passes through $P_9$, and it is this additional base point which makes them ``unexpected''.
\end{example}

The next example shows that the assumption $F\not\equiv 0$ is stronger that the assumption that the points $P_1,\dots,P_r$ impose independent conditions on hypersurfaces of degree $d$. In other words, rows $w(P_1),\dots,w(P_r)$ of the matrix $M$ can have a maximal rank, but the whole matrix has a deficient rank.
\begin{example}\label{exF4}
	Consider the $24$ points in the projective space $\mathbb{P}^3$, coming from the root system $F_4$, which have the coordinates
	$$(1,1,0,0),\; (1,-1,0,0),\; (1,0,1,0),\; (1,0,-1,0),\; (1,0,0,1),\; (1,0,0,-1),\; (0,1,1,0), $$ 
	$$(0,1,-1,0),\; (0,1,0,1),\; (0,1,0,-1),\; (0,0,1,1),\; (0,0,1,-1),\; (1,0,0,0),\; (0,1,0,0),$$
	$$(0,0,1,0),\; (0,0,0,1),\; (1,1,1,1),\; (1,1,-1,1),\; (1,1,1,-1),\; (1,1,-1,-1),\; (1,-1,1,1),$$
	$$ (1,-1,-1,1),\; (1,-1,1,-1),\; (1,-1,-1,-1).$$
	Denote the set of all these points by $Z$. What was shown by authors in \cite[Section 3.6]{H-M-N-T}, if we take $d=4$ and $m=3$, then for any general point $B$ of multiplicity $3$ we have $\dim[I(3B+Z)]_4=4$, while the expected dimension of this system is $\binom{7}{3}-\binom{5}{2}-23=2.$ The number $23$ is the number of independent conditions imposed by points from $Z$ on the space of quartics. 
	
	We want to proceed in the same way as in the previous example and create the interpolation matrix $M$. This matrix should consists of $35$ rows, where $24$ of them are coming from points, $10$ from partial derivatives and the last one holds all monomials of degree $4$. Since all points in $Z$ do not impose independent conditions on form of degree $4$, we are taking the new set $Z'$ with $23$ points from $Z$ and one additional point, which is general enough, to impose independent condition on the space of quartics. It can be easily checked using any computer algebra software that in this situation $\det(M)=F\equiv 0.$
\end{example}


{\bf Acknowledgements}: {\small Farnik was partially supported by National Science Centre, Poland, grant 2018 /28/C/ST1/00339.
	Harbourne was partially supported by Simons Foundation grant \#524858.
	Malara was partially supported by  National Science Centre, Poland, grant 2016/21/N/ST1/01491,
Szpond was partially supported by National Science  Centre, Poland, grant 2018/30/M/ST1/00148,
 Harbourne and	Tutaj-Gasi\'nska were partially supported by National Science Centre, Poland, grant 2017/26/M/ST1/00707.
	Harbourne and Tutaj-Gasi\'nska thank the Pedagogical University of Cracow, the Jagiellonian University 
	and the University of Nebraska for hosting reciprocal visits by Harbourne and Tutaj-Gasi\'nska
	when some of the work on this paper was done. The authors thank Tomasz Szemberg for many fruitful discussions.}

{\tiny \noindent
	Marcin Dumnicki, Halszka Tutaj-Gasi\'nska\\
	Faculty of Mathematics and Computer Science,	Jagiellonian University\\
	{\L}ojasiewicza 6, PL-30-348 Krak\'ow, Poland\\
	Marcin.Dumnicki@im.uj.edu.pl\\
		Halszka.Tutaj@im.uj.edu.pl
	\\
	\\
	\noindent
{\L}ucja Farnik, Grzegorz Malara, Justyna Szpond\\
 Institute of Mathematics, Pedagogical University Cracow\\
Podchor\c{a}\.zych 2, PL-30-084 Krak\'ow, Poland\\
lucja.farnik@gmail.com\\
	grzegorzmalara@gmail.com\\
	szpond@gmail.com
	\\
	\\
	\noindent
	Brian Harbourne\\	Department of Mathematics,	University of Nebraska\\
	Lincoln, NE 68588-0130 USA\\
	bharbourne1@unl.edu

\vspace{0.5cm}
\noindent
Grzegorz Malara current address: Institute of Mathematics, Polish Academy of Sciences, \\
\'Sniadeckich 8, PL-00-656 Warszawa, Poland
}

\end{document}